\documentclass{elsart}

\usepackage{amsmath}
\usepackage{amsfonts}
\usepackage{enumerate}
\usepackage[numbers,sort&compress]{natbib}

\usepackage{pstricks,pst-node}
\newcommand{\vertex}[2]{\cnode*(#2){4pt}{#1}}
\newcommand{\lbvertex}[2]{\cnode*(#2){4pt}{#1}\rput(#2){\white\tiny\textbf{\textsf{#1}}}}
\newcommand{\edge}[2]{\ncline[nodesep=0pt]{-}{#1}{#2}}

\newenvironment{proof}[1][]{\textit{Proof#1}. }{\qed\par}

\newcommand{\Laplacian}{L}

\makeatletter
\def\ps@copyright{\let\@mkboth\@gobbletwo
  \def\@oddhead{}%
  \let\@evenhead\@oddhead
  \def\@oddfoot{\small\slshape
    \def\@tempa{0}
      \hfil\@date\/}%
  \let\@evenfoot\@oddfoot
}
\makeatother

\begin{document}


\begin{frontmatter}

\title{Largest Laplacian Eigenvalue and Degree Sequences of Trees}

\author[IU]{T\"urker B{\i}y{\i}ko\u{g}lu},
\ead{turker.biyikoglu@isikun.edu.tr}
\ead[url]{http://math.isikun.edu.tr/turker/}
\author[UL]{Marc Hellmuth}, and 
\ead{marc@bioinf.uni-leipzig.de}
\ead[url]{http://http://www.bioinf.uni-leipzig.de/\~{}marc/}
\author[WU]{Josef Leydold\corauthref{cor}}
\ead{Josef.Leydold@statistik.wu-wien.ac.at}
\ead[url]{http://statistik.wu-wien.ac.at/\~{}leydold/}

\address[IU]{Department of Mathematics,
  I\c{s}{\i}k University,
  \c{S}ile 34980, Istanbul, Turkey}
\address[UL]{Department of Computer Science, Bioinformatics,
  University of Leipzig, Haertelstrasse 16--18, D-04107 Leipzig, Germany}
\address[WU]{Department of Statistics and Mathematics,
  University of Economics and Business Administration,
  Augasse 2--6, A-1090 Wien, Austria}

\corauth[cor]{Corresponding author. Tel +43 1 313 36--4695. FAX +43 1 313 36--738}



\begin{abstract}
  We investigate the structure of trees that have greatest maximum
  eigenvalue among all trees with a given degree sequence.
  We show that in such an extremal tree the degree sequence is
  non-increasing with respect to an ordering of the vertices that is
  obtained by breadth-first search. This structure is uniquely determined up to
  isomorphism. We also show that the maximum eigenvalue in such 
  classes of trees is strictly monotone with respect to majorization.
\end{abstract}

\end{frontmatter}


\markboth{T.~B{\i}y{\i}ko\u{g}lu, M.~Hellmuth, and J.~Leydold}{%
  Largest Laplacian Eigenvalue and Degree Sequences of Trees}


\section{Introduction}

The \emph{Laplacian} matrix $\Laplacian(G)$ of a graph $G=(V,E)$
with vertex set $V$ and edge set $E$ is given
as 
\begin{equation}
  \Laplacian(G) = D(G) - A(G)\;,
\end{equation}
where $A(G)$ denotes the adjacency matrix of $G$ and
$D(G)$ is the diagonal matrix whose entries are the vertex degrees,
i.e., $D_{vv} = d(v)$, where $d(v)$ denotes the degree of vertex $v$.
We write $\Laplacian$ for short if there is no risk of confusion.

The Laplacian $\Laplacian$ is symmetric and all its eigenvalues
are non-negative. These eigenvalues have been intensively
investigated, see e.g.\ \citep{Mohar:1997a} for a comprehensive
survey. In particular the largest eigenvalue, denoted by $\lambda(G)$
throughout the paper, is of importance. In literature there exist many
bounds on the largest eigenvalue of a graph; in
\citet{Brankov;Hansen;Stevanovic:2005a} some of them are collected and
it is shown how these can be derived in a systematic way.

Here we are interested in the structure of trees which have
largest maximum eigenvalue among all trees with a given degree
sequence. We call such trees \emph{extremal trees}.
We show that for such trees the degree sequence is non-increasing with
respect to an ordering of the vertices that is obtained by breadth-first search.
We also show that the largest maximum eigenvalue in such classes of
trees is strictly monotone with respect to some partial ordering of
degree sequences.
(Similar results hold for the spectral radius of trees with given
degree sequence \citep{Biyikoglu;Leydold:2007b}.)

The paper is organized as follows:
The results of this paper are stated in
Section~\ref{sec:results}. In Section~\ref{sec:proofs} we prove these
theorems by means of a technique of rearranging graphs which has been
developed in \citep{Biyikoglu;Leydold:2006a,Biyikoglu;Leydold;Stadler:2007a}
for the problem of
minimizing the first Dirichlet eigenvalue within a class of trees. 


\section{Degree Sequences and Largest Eigenvalue}
\label{sec:results}

Let $d(v)$ denote the degree of vertex $v$. We call a vertex $v$ with
$d(v)=1$ a \emph{pendant vertex} of a graph (or \emph{leaf} in case
of a tree).
Recall that a sequence $\pi=(d_0,\ldots,d_{n-1})$ of non-negative
integers is called \emph{degree sequence} if there exists a graph $G$
for which $d_0,\ldots,d_{n-1}$ are the degrees of its vertices. 
In particular, $\pi$ is a tree sequence, i.e.\ a degree sequence of
some tree, if and only if every $d_i>0$ and $\sum_{i=0}^{n-1}
d_i=2\,(n-1)$. We refer the reader to \citep{Melnikov;etal:1994a} for
relevant background on degree sequences.
We introduce the following class for which we can characterize
extremal graphs with respect to the maximum eigenvalue.
\begin{equation*}
  \mathcal{T}_\pi 
  = \{\text{$G$ is a tree with degree sequence $\pi$}\}\,.
\end{equation*}

For this characterization of extremal trees in $\mathcal{T}_\pi$ we
introduce an ordering of the vertices $v_0,\ldots,v_{n-1}$ of a graph
$G$ by means of breadth-first search:
Select a vertex $v_0\in G$ and create a sorted list of vertices
beginning with $v_0$; append all neighbors $v_1,\ldots,v_{d(v_0)}$ of
$v_0$ sorted by decreasing degrees; then append all neighbors of $v_1$
that are not already in this list; continue recursively with
$v_2,v_3,\ldots$ until all vertices of $G$ are processed.
In this way we build layers where 
each $v$ in layer $i$ is adjacent to some vertex $w$ in
layer $i-1$ and vertices $u$ in layer $i+1$.
We then call the vertex $w$ the \emph{parent} of $v$ and $v$ a child
of $w$.

\begin{defn}[BFD-ordering]
  Let $G(V,E)$ be a connected graph with root $v_0$. Then a
  well-ordering $\prec$ of the vertices is called \emph{breadth-first
  search ordering with decreasing degrees} (\emph{BFD}-ordering for
  short) if the following holds for all vertices $v, w\in V$:
  \begin{enumerate}[(B1)]
  \item if $w_1\prec w_2$ then $v_1\prec v_2$ for all children $v_1$
    of $w_1$ and $v_2$ of $w_2$;
  \item if $v\prec u$, then $d(v)\geq d(u)$.
  \end{enumerate}
  We call connected graphs that have a BFD-ordering of its vertices a
  \emph{BFD-graph} (see Fig.~\ref{fig:BFD-tree} for an example).
\end{defn}

\begin{figure}[ht]
  \centering
  \psset{unit=10mm}
  \begin{pspicture}(-3.5,0.5)(2.4,-3)
    \lbvertex{0}{0,0}
    \lbvertex{1}{-2.4,-1}
    \lbvertex{2}{-0.8,-1}
    \lbvertex{3}{0.8,-1}
    \lbvertex{4}{2.4,-1}
    \lbvertex{5}{-3,-2}
    \lbvertex{6}{-2.4,-2}
    \lbvertex{7}{-1.8,-2}
    \lbvertex{8}{-1.1,-2}
    \lbvertex{9}{-0.5,-2}
    \lbvertex{10}{0.5,-2}
    \lbvertex{11}{1.1,-2}
    \lbvertex{12}{2.1,-2}
    \lbvertex{13}{2.7,-2}
    \lbvertex{14}{-3.5,-3}
    \lbvertex{15}{-3,-3}
    \lbvertex{16}{-2.4,-3}
    \lbvertex{17}{-1.8,-3}
    \lbvertex{18}{-1.1,-3}
    \edge{0}{1}
    \edge{0}{2}
    \edge{0}{3}
    \edge{0}{4}
    \edge{1}{5}
    \edge{1}{6}
    \edge{1}{7}
    \edge{2}{8}
    \edge{2}{9}
    \edge{3}{10}
    \edge{3}{11}
    \edge{4}{12}
    \edge{4}{13}
    \edge{5}{14}
    \edge{5}{15}
    \edge{6}{16}
    \edge{7}{17}
    \edge{8}{18}
  \end{pspicture}
  \caption{A BFD-tree with degree sequence $\pi=(4^2,3^4,2^3,1^{10})$}
  \label{fig:BFD-tree}
\end{figure}
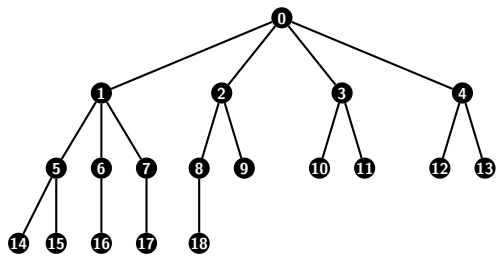

Every graph has for each of its vertices $v$ an ordering with root
$v$ that satisfies (B1). This can be found by a breadth-first
search as described above. However, not all trees have an ordering
that satisfies both (B1) and (B2); consider the tree in
Fig.~\ref{fig:non-BFD-tree}.

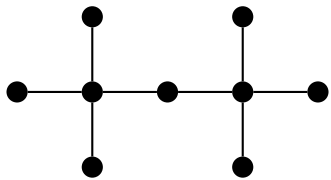
\begin{figure}[ht]
  \centering
  \psset{unit=10mm}
  \begin{pspicture}(-2,-1.3)(2,1.3)
    \vertex{0}{0,0}
    \vertex{1}{1,0}
    \vertex{2}{-1,0}
    \vertex{11}{2,0}
    \vertex{12}{1,1}
    \vertex{13}{1,-1}
    \vertex{21}{-2,0}
    \vertex{22}{-1,1}
    \vertex{23}{-1,-1}
    \edge{0}{1}
    \edge{0}{2}
    \edge{1}{11}
    \edge{1}{12}
    \edge{1}{13}
    \edge{2}{21}
    \edge{2}{22}
    \edge{2}{23}
  \end{pspicture}
  \caption{A tree with degree sequence $\pi=(4^2,2^1,1^6)$ where no
    BFD-ordering exists.}
  \label{fig:non-BFD-tree}
\end{figure}

\begin{thm}
  \label{thm:Tdegseq}
  A tree $G$ with degree sequence $\pi$ is extremal in class
  $\mathcal{T}_\pi$ if and only if it is a BFD-tree. $G$ is then
  uniquely determined up to isomorphism.
  The BFD-ordering is consistent with the corresponding eigenvector
  $f$ of $G$ in such a way that $|f(u)|>|f(v)|$ implies $u\prec v$.
\end{thm}

For a tree with degree sequence $\pi$ a sharp upper bound on the
largest eigenvalue can be found by computing the corresponding 
BFD-tree. Obviously finding this tree can be done 
in $O(n)$ time if the degree sequence is sorted.

We define a partial ordering on degree sequences 
$\pi=(d_0,\ldots,d_{n-1})$ and $\pi'=(d'_0,\ldots,d'_{n'-1})$ with
$n\leq n'$ and $\pi\not=\pi'$ as follows:
we write $\pi\lhd\pi'$ if and only if 
$\sum_{i=0}^j d_i\leq \sum_{i=0}^j d'_i$
for all $j=0,\ldots n-1$ (recall that the degree sequences are
non-increasing).
\begin{thm}
  \label{thm:monotone}
  Let $\pi$ and $\pi'$ be two distinct degree sequences of trees with
  $\pi\lhd\pi'$. Let $G$ and $G'$ be extremal trees in the classes
  $\mathcal{T}_\pi$ and $\mathcal{T}_{\pi'}$, respectively.
  Then we find for the corresponding maximum eigenvalues 
  $\lambda(G)<\lambda(G')$. 
\end{thm}

We get the following well-known results as immediate corollaries.

\begin{cor}
  \label{thm:Tstar}
  A tree $T$ is extremal in the class of all trees
  with $n$ vertices if and only if it is the star $K_{1,n-1}$.
\end{cor}

\begin{cor}[{\citep[Thm.~2.2]{Hong;Zhang:2005a}}]
  \label{thm:Tlongstar}
  A tree $G$ is extremal in the class of all trees
  with $n$ vertices and $k$ leaves if and only if it is a star with paths 
  of almost the same lengths attached to each of its $k$ leaves.
\end{cor}
\begin{proof}[ of Cor.~\ref{thm:Tstar} and \ref{thm:Tlongstar}]
  The tree sequences $\pi_n=(n-1,1,\ldots,1)$ and
  $\pi_{n,k}=(k,2,\ldots,2,1,\ldots,1)$ are maximal w.r.t.\ ordering
  $\lhd$ in the respective classes of all trees with $n$ vertices and
  all trees with $n$ vertices and $k$ pendant vertices.
  Thus the statement immediately follows from
  Theorems~\ref{thm:Tdegseq} and \ref{thm:monotone}.
\end{proof}


\section{Proof of the Theorems}
\label{sec:proofs}

We denote the largest eigenvalue of a graph $G$ by
$\lambda(G)$. We denote the number of vertices of a graph $G$ by
$n=|V|$ and the geodesic path between two vertices $u$ and
$v$ by $P_{uv}$.

The Rayleigh quotient of the graph Laplacian $\Laplacian$
of a vector $f$ on $V$ is the fraction 
\begin{equation}
  \label{eq:rayleigh}
  \mathcal{R}_G(f)
  = \frac{\langle f,L f\rangle}{\langle f,f \rangle}
  = \frac{\sum_{uv\in E} (f(u)-f(v))^2}{\sum_{v\in V} f(v)^2}\;.
\end{equation}

By the Rayleigh-Ritz Theorem we find the following well-known
property for the spectral radius of $G$.
\begin{prop}[\citep{Horn:1990a}]
  \label{prop:Rayleigh-Ritz}
  Let $\mathcal{S}$ denote the set of unit vectors on $V$.
  Then 
  \begin{equation*}
    \lambda(G) 
    = \max_{f\in\mathcal{S}} \mathcal{R}_G(f)
    = \max_{f\in\mathcal{S}} \sum_{uv\in E} (f(u)-f(v))^2\;.
  \end{equation*}
  Moreover, if $\mathcal{R}_G(f)=\lambda(G)$ for a 
  function $f\in\mathcal{S}$, then $f$ is the Laplacian eigenvector
  corresponding to the maximum eigenvalue $\lambda(G)$ of $L(G)$.
\end{prop}

Notice that every eigenvector $f$ corresponding to the maximum
eigenvalue must fulfill the eigenvalue equation 
\begin{equation}
  \label{eq:eigenvalue-eq}
  \Laplacian f(x) = d(x) f(x) - \sum_{yx\in E} f(y) = \lambda f(x)
  \quad\text{for every $x\in V$.}
\end{equation}

Trees are a special case of bipartite graphs. Hence the following
observation is important.
\begin{prop}[\citep{Roth:1989a}]
  \label{thm:bipartite}
  Let $G(V_1\cup V_2,E)$ be a connected graph with bipartition
  $V_1\cup V_2$ and $n=|V_1\cup V_2|$ vertices. 
  Then there is an eigenfunction $f$ corresponding to the maximum
  eigenvalue of $L(G)$, such that $f$ is positive on $V_1$ and
  negative on $V_2$.
\end{prop}

The main techniques for proving our theorems is \emph{rearrangement} of
edges. We need two types of rearrangement steps that we call
\emph{switching} and \emph{shifting}, resp., in the following.

\begin{lem}[\emph{Switching}]
  \label{lem:switching}
  Let $T\in \mathcal{T}_{\pi}$ and let $u_1v_1$, $u_2v_2 \in E(T)$ 
  be edges such that the path $P_{v_1v_2}$ neither contains $u_1$
  nor $u_2$.
  Then by replacing edges $u_1v_1$ and $u_2v_2$ by the respective 
  edges $u_1v_2$ and $u_2v_1$ we get a new tree $T'$ which is also
  contained in $\mathcal{T}_\pi$. 
  Furthermore for every eigenvector $f$ corresponding to the maximum
  eigenvalue $\lambda(T)$ we find 
  $\lambda(T')\geq\lambda(T)$ whenever $|f(u_1)| \geq |f(u_2)|$ and 
  $|f(v_2)| \geq |f(v_1)|$. The inequality is strict if 
  one of the latter two inequalities is strict.
\end{lem}
\begin{proof}
  Since $P_{v_1v_2}$ neither contains $u_1$ nor $u_2$ by assumption,
  $T'$ is again a tree. Since switching of two edges does not change
  degrees, $T'$ also belongs to class $\mathcal{T}_{\pi}$.
  Let $f$ be an eigenvector corresponding to the maximum eigenvalue
  $\lambda(T)$ with $||f||=1$. Without loss of generality we assume
  that $f(v_1)>0$.
  To verify the inequality we have to compute the effects of removing
  and inserting edges on the Rayleigh quotient. We have to distinguish
  between two cases: 
  \begin{enumerate}[(1)]
  \item 
    $f(v_1)$ and $f(u_2)$ have different signs. Then by
    Prop.~\ref{thm:bipartite} and our assumptions
    $0<f(v_1)\leq f(v_2)$ and $f(u_1)\leq f(u_2) <0$. Thus 
    \[
    \begin{split}
      \lambda(T') - \lambda(T) 
      &\geq \langle f,\Laplacian(T')f\rangle - \langle f,\Laplacian(T)f\rangle \\
      &= \left[(f(u_1)-f(v_2))^2 + (f(v_1)-f(u_2))^2\right] \\
      &\quad - \left[(f(v_1)-f(u_1))^2 + (f(u_2)-f(v_2))^2\right] \\
      &=2(f(u_1)-f(u_2))(f(v_1)-f(v_2)) \\
      &\geq 0\;.
    \end{split}
    \]
  \item 
    $f(v_1)$ and $f(u_2)$ have the same sign. 
    Then $0<f(v_1)\leq -f(v_2)$ and $0< f(u_2) \leq -f(u_1)$.
    We define a new function $f'$ such that $f'(x)=f(x)$ for all $x$
    that belong to the same component of $T\setminus\{v_1u_1,v_2u_2\}$
    as $v_1$ and $v_2$, and $f'(x)=-f(x)$ otherwise. Thus
    \[
    \begin{split}
      \lambda(T') - \lambda(T) 
      &\geq \langle f',\Laplacian(T')f'\rangle - \langle f,\Laplacian(T)f\rangle \\
      &= \left[(f'(u_1)-f'(v_2))^2 + (f'(v_1)-f'(u_2))^2\right] \\
      &\quad - \left[(f(v_1)-f(u_1))^2 + (f(u_2)-f(v_2))^2\right] \\
      &=2(f(u_1)+f(u_2))(f(v_1)+f(v_2)) \\
      &\geq 0\;.
    \end{split}
    \]
  \end{enumerate}
  Therefore in both cases $\lambda(T')\geq\lambda(T)$. 
  If $|f(u_1)| > |f(u_2)|$ or $|f(v_2)| > |f(v_1)|$ then the 
  eigenvalue equation~(\ref{eq:eigenvalue-eq}) would not 
  hold for $v_1$ or $u_2$. Thus $f$ (and $f'$, resp.)\ is not an
  eigenfunction corresponding to $\lambda(T')$ and thus
  $\lambda(T')>\mathcal{R}_{T'}(f')\geq \lambda(T)$ as claimed.
\end{proof}

\begin{lem}[\emph{Shifting}]
  \label{lem:shifting}
  Let $T\in \mathcal{T}_{\pi}$ and $u,v\in V(T)$.
  Assume we have edges $ux_1,\ldots,ux_k\in E(T)$ such that none of the
  $x_i$ is in $P_{uv}$. Then we get a new
  graph $T'$ by replacing all edges $ux_1,\ldots,ux_k$ by the
  respective edges $vx_1,\ldots,vx_k$. 
  If $f$ is an eigenvector with respect to $\lambda(T)$,
  then we find $\lambda(T') > \lambda(T)$ whenever 
  $|f(u)|\leq |f(v)|$.
\end{lem}
\begin{proof}
  Assume without loss of generality that $f(u)>0$. Then we have two
  cases:
  $f(v)$ and $f(u)$ have the same sign. Then by our assumptions
  $f(v)\geq f(u) > 0$ and $f(x_i)<0$ for all $i=1,\ldots,k$, and we
  find
  \[
  \begin{split}
    \lambda(T') - \lambda(T) 
    &\geq \langle f,\Laplacian(T')f\rangle - \langle f,\Laplacian(T)f\rangle \\
    &= \sum_{i=1}^k \left[(f(x_i)-f(v))^2 - (f(x_i)-f(u))^2\right] \\
    &=2 (f(u)-f(v)) \sum_{i=1}^k f(x_i) + k(f(v)^2-f(u)^2)\\
    &\geq 0\;.
  \end{split}
  \]
  Now if $\lambda(T')=\lambda(T)$ then $f$ also must be an eigenvector
  of $T'$ by Prop.~\ref{prop:Rayleigh-Ritz}. Thus the eigenvalue
  equation~(\ref{eq:eigenvalue-eq}) for vertex $u$ and $v$ in $T$ and
  $T'$ implies that $f(x_i)=0$ for all $i$, a contradiction.
  The second case where $f(v)$ and $f(u)$ have different signs is
  shown by means of a function $f'$ analogously to the proof of
  Lemma~\ref{lem:switching}.
\end{proof}

\begin{lem}
  \label{lem:degree}
  Let $T$ be extremal in class $\mathcal{T}_{\pi}$ and $f$ an
  eigenvector corresponding to $\lambda(T)$.
  If $|f(v)|>|f(u)|$, then $d(v)\geq d(u)$.
\end{lem}
\begin{proof}
  Suppose that $|f(v)|>|f(u)|$ for some vertices $u,v\in V(T)$ but 
  $d(v) < d(u)$. Then we construct a new graph $T'\in\mathcal{T}_{\pi}$
  by shifting $k=d(u)-d(v)$ edges in $T$. For this purpose we can
  choose any $k$ of the $d(u)-1$ edges that are not contained in
  $P_{uv}$. Let $x_1u,\ldots, x_ku$ be these
  edges which are replaced by $x_1v,\ldots, x_kv$. Thus we can apply
  Lemma~\ref{lem:shifting} and obtain
  $\lambda(T')>\lambda(T)$, a contradiction to our assumption.
\end{proof}

\begin{lem}
  \label{lem:BFD-unique}
  Each class $\mathcal{T}_{\pi}$ contains a BFD-tree $T$ that is
  uniquely determined up to isomorphism.
\end{lem}
\begin{proof}
  For a given tree sequence the construction of a BFD-tree is
  straightforward.
  To show that two BFD-trees $T$ and $T'$ in class $\mathcal{T}_{\pi}$ 
  are isomorphic we use a function $\phi$ that maps the vertex $v_{i}$ in the
  $i$th position in the BFD-ordering of $T$ to the vertex $w_{i}$ in the
  $i$th position in the BFD-ordering of $T'$. By the properties (B1) and (B2)
  $\phi$ is an isomorphism, as $v_{i}$ and $w_{i}$ have the same degree and the
  images of neighbors of $v_{i}$ in the next layer are exactly the neighbors
  of $w_{i}$ in the next layer. The latter can be seen by looking on all
  vertices of $T$ in the reverse BFD-ordering. 
\end{proof}

Now let $f$ be an eigenvector corresponding to the maximum eigenvalue
$\lambda(T)$ of $T$. Then we can define an ordering $\prec$ of the vertices
of $T$ in such a way that $v_i\prec v_j$ whenever
\begin{enumerate}[(i)]
\item $|f(v_i)|>|f(v_j)|$ or
\item $|f(v_i)|=|f(v_j)|$ and $d(v_i)>d(v_j)$ or
\item $|f(v_i)|=|f(v_j)|$, $d(v_i)=d(v_j)$, and there is a neighbor
  $u_i$ of $v_i$ with $u_i\prec u_j$ for all neighbors $u_j$ of $v_j$.
\end{enumerate}
Such an ordering can always be constructed recursively
starting at a maximum $v_0$ of $|f(x)|$. If we have already enumerated
the vertices in $V_{k-1}=\{v_0,\ldots,v_{k-1}\}$ then the next vertex
$v_k$ is the maximum of $V\setminus V_{k-1}$ w.~r.~t.\ (i) and
(ii). If $v_k$ is not uniquely determined then we look at the
respective neighbors that belong to $V_{k-1}$ and select the vertex
with the least neighbor (in the ordering of $V_{k-1}$). It might
happen that this is still not uniquely determined or that there are no
such neighbors, then we are free to choose any of the qualified
vertices.

We enumerate the vertices of $T$ with respect to this ordering,
i.e., $v_i\prec v_j$ if and only if $i<j$.
In particular, $v_0$ is a maximum of $|f|$.

\begin{lem}
  \label{lem:BFD-exists}
  Let $T$ be extremal in class $\mathcal{T}_\pi$
  with corresponding eigenvector $f$. Then the order $\prec$
  defined above is a BFD-ordering.
\end{lem}
\begin{proof}
  Property (B2) immediately follows from Lemma~\ref{lem:degree}.
  Let $v_0$ be the root of $T$ and create another
  ordering of its vertices by a breadth-first search where the
  children of a vertex are always sorted by their index. 
  We denote the $i$th element with respect to this ordering by
  $v(i)$. 
  We show that both orderings are equivalent, i.e.,
  $v(i)=v_i$.
  Suppose that there exists an index $k$ where this relation fails and
  choose $k$ the least index with this property. Then 
  $v(k)=v_m\succ v_k$ and 
  consequently $|f(v_k)|\geq|f(v_m)|$ and $d(v_k)\geq d(v_m)$.
  Let $w_m$ and $w_k$ be the respective parents of $v_m$ and $v_k$.
  Notice that $w_m\prec v_k$ and $w_k\succ w_m$ since
  $(v_0,\ldots,v_{k-1},v_m)$ is already a BFD-ordering by our construction. 
  We have the following cases:
  \begin{enumerate}[(1)]
  \item $|f(v_k)| > |f(v_m)|$ and the path $P_{w_m w_k}$ does not
    contain any of the two vertices $v_k$ and 
    $v_m$. Then we can replace edges $w_kv_k$ and $w_mv_m$ by $w_mv_k$
    and $w_kv_m$ and get a new tree $T'$ with $\lambda(T')>\lambda(T)$ by
    Lemma~\ref{lem:switching}.
  \item $|f(v_k)| > |f(v_m)|$ and $v_m$ is contained in $P_{w_m w_k}$.
    Then by Lemma~\ref{lem:degree}
    $d(v_k)\geq d(v_m)\geq 2$ and there exists a child $u_k$ of
    $v_k$. By construction $u_k\succ v_k\succ w_m$.
    Again we get a new tree $T'$ by replacing edges 
    $w_mv_m$, $u_kv_k$ by the edges $w_mv_k$, $u_kv_m$, with
    $\lambda(T')>\lambda(T)$.
    Notice that $v_k$ cannot be in the path $P_{w_m w_k}$.
  \item $|f(v_k)| = |f(v_m)|$ and $d(v_k)>d(v_m)$. Then we can shift
    $k=d(v_k)-d(v_m)$ children of $v_k$ and get a new tree
    $T'\in\mathcal{T}_\pi$ with $\lambda(T')>\lambda(T)$ by
    Lemma~\ref{lem:shifting}.
  \item $|f(v_k)|=|f(v_m)|$ and $d(v_k)=d(v_m)$. But then we had
    $w_k\prec w_m$, a contradiction to (iii) of our ordering.
  \end{enumerate}
  In either case we get a contradiction to our assumption that
  $T$ is already extremal.
\end{proof}

\begin{proof}[ of Theorem~\ref{thm:Tdegseq}]
  The result follows immediately from the construction of the ordering
  $\prec$ and Lemmata~\ref{lem:BFD-unique} and \ref{lem:BFD-exists}.
\end{proof}

\begin{proof}[ of Theorem~\ref{thm:monotone}]
  Let $\pi=(d_0,\dots,d_{n-1})$ and $\pi'=(d'_0,\dots,d'_{n'-1})$ be
  two tree sequences with $\pi\lhd\pi'$ and $n=n'$. 
  By Theorem~\ref{thm:Tdegseq} the maximum eigenvalue becomes the largest one for
  a tree $T$ within class $\mathcal{T}_\pi$ when $T$ is a BFD-tree.
  Again $f$ denotes the eigenvector affording $\lambda(T)$.
  We have to show that there exists a tree $T'\in\mathcal{T}_{\pi'}$
  such that $\lambda(T')>\lambda(T)$. Therefore we construct a
  sequence of trees $T=T_0\rightarrow T_1\rightarrow \dots \rightarrow
  T_{s}=T'$ by shifting edges and show that the
  $\lambda(T_j)>\lambda(T_{j-1})$ for every $j=1,\dots,s$.
  We denote the degree sequence of $T_j$ by $\pi^{(j)}$.

  For a particular step in our construction, let $k$ be the least index
  with $d'_k > d^{(j)}_k$. Let $v_k$ be the corresponding
  vertex in tree $T_j$. Since 
  $\sum_{i=0}^k d'_i > \sum_{i=0}^k d^{(j)}_i$
  and $\sum_{i=0}^{n-1} d'_i = \sum_{i=0}^{n-1} d^{(j)}_i = 2(n-1)$
  there must exist a vertex $v_l\succ v_k$ with degree $d^{(j)}_l \geq 2$.
  Thus it has a child $u_l$. By Lemma~\ref{lem:shifting} we can replace
  edge $v_lu_l$ by edge $v_ku_l$ and get a new tree $T_{j+1}$ with
  $\lambda(T_{j+1})>\lambda(T_j)$. 
  Moreover, $d^{(j+1)}_k = d^{(j)}_k+1$ and $d^{(j+1)}_l = d^{(j)}_l-1$,
  and consequently $\pi^{(j)}\lhd\pi^{(j+1)}$.
  By repeating this procedure we end up with degree sequence $\pi'$ and
  the statement follows for the case where $n'=n$.

  Now assume $n'>n$. Then we construct a sequence of trees $T_j$ by the
  same procedure. However, now it happens that we arrive at some tree
  $T_r$ where $d'_k > d^{(r)}_k$ but $d^{(r)}_l=1$ for all $v_l\succ
  v_k$, i.e., they are pendant vertices. In this case we join a new
  pendant vertex to $v_k$. Then $d^{(r+1)}_k = d^{(r)}_k+1$ and
  $|\pi^{(r+1)}|=|\pi^{(r)}|+1$ as we have added a new vertex degree of
  value $1$. Thus $\pi^{(r+1)}$ is again a tree sequence with 
  $\pi^{(r)}\lhd\pi^{(r+1)}$. 
  Moreover, $\lambda(T_{r+1})>\lambda(T_r)$ as $T_{r+1}\supset
  T_r$.
  By repeating this procedure we end up with degree sequence $\pi'$
  and the statement of Theorem~\ref{thm:monotone}.
\end{proof}


\section{Addendum to the Proof}
\label{sec:addendum-proof}

This manuscript has been compiled while M.H.\ visited Vienna last
summer (2007). We applied methods developed in 
\citep{Biyikoglu;Leydold:2006a} and \citep{Biyikoglu;Leydold:2007b}.
Meanwhile \citet{Zhang:2007a} has published the same results. Thus we
decided to present our proof to interested readers by this technical
report. 



\begin{thebibliography}{10}
\providecommand{\natexlab}[1]{#1}
\providecommand{\url}[1]{\texttt{#1}}
\expandafter\ifx\csname urlstyle\endcsname\relax
  \providecommand{\doi}[1]{doi: #1}\else
  \providecommand{\doi}{doi: \begingroup \urlstyle{rm}\Url}\fi

\bibitem[B{\i}y{\i}ko{\u{g}}lu and Leydold(2006)]{Biyikoglu;Leydold:2006a}
T.~B{\i}y{\i}ko{\u{g}}lu and J.~Leydold.
\newblock {Faber-Krahn} type inequalities for trees.
\newblock \emph{J. Comb. Theory, Ser. B}, 97\penalty0 (2):\penalty0 159--174,
  2006.

\bibitem[B{\i}y{\i}ko{\u g}lu and Leydold(2007)]{Biyikoglu;Leydold:2007b}
T.~B{\i}y{\i}ko{\u g}lu and J.~Leydold.
\newblock Graphs with given degree sequence and maximal spectral radius.
\newblock preprint, 2007.
\newblock submited to Elec. J. Comb.

\bibitem[B{\i}y{\i}ko{\u{g}}lu et~al.(2007)B{\i}y{\i}ko{\u{g}}lu, Leydold, and
  Stadler]{Biyikoglu;Leydold;Stadler:2007a}
T.~B{\i}y{\i}ko{\u{g}}lu, J.~Leydold, and P.~F. Stadler.
\newblock \emph{Laplacian Eigenvectors of Graphs. Perron-Frobenius and
  Faber-Krahn Type Theorems}, volume 1915 of \emph{Lecture Notes in
  Mathematics}.
\newblock Springer, 2007.

\bibitem[Brankov et~al.(2006)Brankov, Hansen, and
  Stevanovi{\'c}]{Brankov;Hansen;Stevanovic:2005a}
V.~Brankov, P.~Hansen, and D.~Stevanovi{\'c}.
\newblock Automated conjectures on upper bounds for the largest laplacian
  eigenvalue of graphs.
\newblock \emph{Lin. Algebra Appl.}, 414\penalty0 (2--3):\penalty0 407--424,
  2006.
\newblock \doi{10.1016/j.laa.2005.10.017}.

\bibitem[Hong and Zhang(2005)]{Hong;Zhang:2005a}
Y.~Hong and X.-D. Zhang.
\newblock Sharp upper and lower bounds for largest eigenvalue of the laplacian
  matrices of trees.
\newblock \emph{Discr. Math.}, 296\penalty0 (2--3):\penalty0 187--197, 2005.

\bibitem[Horn and Johnson(1990)]{Horn:1990a}
R.~A. Horn and C.~R. Johnson.
\newblock \emph{Matrix Analysis. Reprinted with corrections}.
\newblock Cambridge University Press, 1990.

\bibitem[Melnikov et~al.(1994)Melnikov, Tyshkevich, Yemelichev, and
  Sarvanov]{Melnikov;etal:1994a}
O.~Melnikov, R.~I. Tyshkevich, V.~A. Yemelichev, and V.~I. Sarvanov.
\newblock \emph{Lectures on Graph Theory}.
\newblock B.I. Wissenschaftsverlag, Mannheim, 1994.
\newblock Transl. from the Russian by N. Korneenko with the collab. of the
  authors.

\bibitem[Mohar(1997)]{Mohar:1997a}
B.~Mohar.
\newblock Some applications of {L}aplace eigenvalues of graphs.
\newblock In G.~Hahn and G.~Sabidussi, editors, \emph{Graph Symmetry: Algebraic
  Methods and Applications}, volume 497 of \emph{NATO ASI Series C:
  Mathematical and Physical Sciences}, pages 225--275. Kluwer Academic
  Publishers, 1997.

\bibitem[Roth(1989)]{Roth:1989a}
R.~Roth.
\newblock On the eigenvectors belonging to the minimum eigenvalue of an
  essentially nonnegative symmetric matrix with bipartite graph.
\newblock \emph{Lin. Algebra Appl.}, 118:\penalty0 1--10, 1989.

\bibitem[Zhang(2007)]{Zhang:2007a}
X.-D. Zhang.
\newblock The {Laplacian} spectral radii of trees with degree sequences.
\newblock \emph{Discrete Mathematics}, to appear, 2007.
\newblock \doi{10.1016/j.disc.2007.06.017}.

\end{thebibliography}


\end{document}